\RequirePackage[l2tabu, orthodox]{nag}

\documentclass[12pt]{amsart}
\usepackage{fullpage,url,amssymb,enumerate,colonequals, graphicx, wrapfig, mathtools}
 \usepackage[all]{xy} 
\usepackage{mathrsfs} 
\usepackage{IEEEtrantools} 
\usepackage{enumitem} 
\usepackage[OT2,T1]{fontenc}
\DeclareSymbolFont{cyrletters}{OT2}{wncyr}{m}{n}
\DeclareMathSymbol{\Sha}{\mathalpha}{cyrletters}{"58}

\usepackage{color}

\newcommand{\F}{\mathbb{F}}

\newcommand{\PP}{\mathbb{P}}
\newcommand{\Q}{\mathbb{Q}}

\newcommand{\Z}{\mathbb{Z}}

\newcommand{\Kbar}{{\overline{K}}}
\newcommand{\ksep}{{k^{\operatorname{sep}}}}
\newcommand{\Ksep}{{K^{\operatorname{sep}}}}

\DeclareMathOperator{\Aut}{Aut}

\DeclareMathOperator{\Char}{char}

\DeclareMathOperator{\disc}{disc}

\DeclareMathOperator{\Gal}{Gal}

\DeclareMathOperator{\sign}{sign}

\newcommand{\GL}{\operatorname{GL}}

\newcommand{\slantsf}[1]{\textsl{\textsf{#1}}}

\newtheorem{theorem}{Theorem}[section]
\newtheorem{lemma}[theorem]{Lemma}
\newtheorem{corollary}[theorem]{Corollary}
\newtheorem{proposition}[theorem]{Proposition}

\theoremstyle{definition}

\newtheorem{conjecture}[theorem]{Conjecture}
\newtheorem{example}[theorem]{Example}

\theoremstyle{remark}
\newtheorem{remark}[theorem]{Remark}

\usepackage{microtype}

\usepackage[
	backref,
	pdfauthor={Bjorn Poonen}, 
]{hyperref}
\usepackage[alphabetic,backrefs,lite]{amsrefs} 

\begin{document}

\title[Large arboreal]{Large arboreal Galois representations}
\keywords{Arithmetic dynamics; arboreal representation; iterated monodromy group}
\author{Borys Kadets}

\address{Department of Mathematics, Massachusetts Institute of Technology, Cambridge, MA 02139-4307, USA}
\email{bkadets@mit.edu}
\urladdr{\url{http://math.mit.edu/~bkadets/}}

\begin{abstract}
Given a field $K$, a polynomial $f \in K[x]$ of degree $d$, and a suitable element $t \in K$, the set of preimages of $t$ under the iterates $f^{\circ n}$ carries a natural structure of a $d$-ary tree. We study conditions under which the absolute Galois group of $K$ acts on the tree by the full group of automorphisms. When $d \geq 20$ is even and $K=\mathbb{Q}$ we exhibit examples of polynomials with maximal Galois action on the preimage tree, partially affirming a conjecture of Odoni. We also study the case of $K=F(t)$ and $f \in F[x]$ in which the corresponding Galois groups are the monodromy groups of the ramified covers $f^{\circ n}: \PP^1_F \to \PP^1_F$.
\end{abstract}

\maketitle

\section{Introduction}

Let $K$ be a field, $f \in K[x]$ be a polynomial of degree $d$, and $t \in K$ be an arbitrary element. Write $f^{\circ n} = f \circ f \circ ... \circ f$ for the $n$th iterate of $f$ (with the convention $f^{\circ 0}(x)=x$). 
Assume that the polynomials $f^{\circ n}(x)-t$ are separable for all $n$. Fix a separable closure $\Ksep$ of $K$. Then the roots of $f^{\circ n}(x) - t$ in $\Ksep$ for varying $n$ have a natural tree structure.
 Namely, define a graph with the set of vertices equal to the union of roots of $f^{\circ n} (x) - t$ for all $n\geq 0$. Draw an edge between two roots $\alpha, \beta$ when $f(\alpha)=\beta$. Then when $t$ is not a periodic point of $f$, the resulting graph is a complete rooted $d$-ary tree $T_\infty$. Vertices at level $n$ correspond to the roots of $f^{\circ n}(x)-t$ (the root of the tree $t$ is at level zero). Let $T_n$ be the tree formed by the first $n$ levels of $T_\infty$. For example, the next figure is the tree $T_2$ when $f=x^2-2$ and $t=0$.

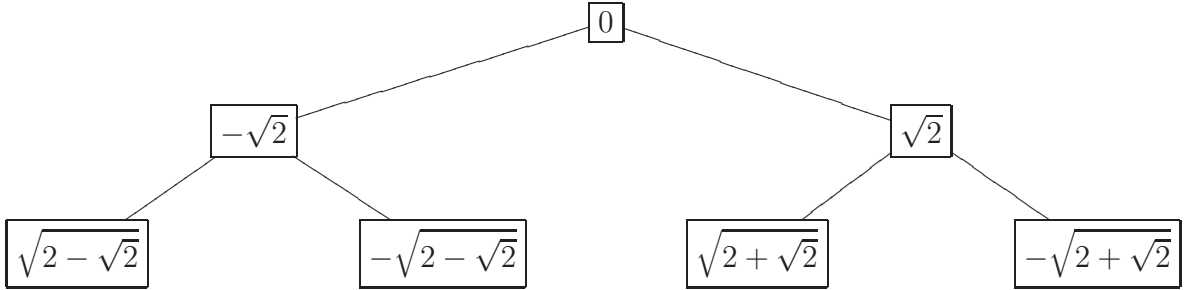
\centerline{
\xymatrix{
&  & &*+[F]{0} \ar@{-}[drr] \ar@{-}[dll] & && \\
  &*+[F]{ -\sqrt{2}} \ar@{-}[dr] \ar@{-}[dl] & & & &*+[F]{\sqrt{2}} \ar@{-}[dr] \ar@{-}[dl] & \\
 *+[F]{\sqrt{2-\sqrt{2}}} &  &*+[F]{-\sqrt{2-\sqrt{2}}} & &*+[F]{\sqrt{2+\sqrt{2}}} & &*+[F]{-\sqrt{2+\sqrt{2}}} 
}}

\vspace{1mm}
Let $G \colonequals \Gal(\Ksep/K)$ be the absolute Galois group of $K$. The Galois action on $\Ksep$ defines a homomorphism $\phi: G \to \Aut(T_\infty)$ known as the \slantsf{arboreal Galois representation} attached to $f$ and $t$. This construction is parallel to that of the Tate module of an abelian variety, where one replaces $f$ by the multiplication by $p$ endomorphism $[p]$. With this analogy in mind, it is natural to seek a counterpart of Serre's open image theorem in arithmetic dynamics. The crucial difference is that while the closed subgroups of $\GL_2(\Z_p)$ are easy to describe using Lie theory, the subgroup structure of the profinite group $\Aut(T_\infty)$ is complicated; this makes analyzing possibilities for the image of an arboreal representation a hard problem. Our goal is to study conditions for having the simplest kind of Galois image, namely the whole group $\Aut(T_\infty)$.  

 \begin{example}
 Let $K=\mathbb{Q}$, $t=0$ and $f=x^2+1$. Then the arboreal representation is surjective (see \cite{Stoll1992}).
 \end{example}
 
It is conjectured that when $K$ is a number field and $d=2$ the image of an arboreal representation has finite index in $\Aut(T_\infty)$ unless some degeneracy conditions are satisfied (see \cite{Jones2013}*{Conjecture~3.11}). However very little is known when the degree of $f$ is greater than $2$. We will show that when the degree of $f$ is even there is a criterion for the surjectivity of the arboreal representation. 
In Section \ref{large} we prove the following theorem.

\begin{theorem}\label{intro-surjective}
Assume that the degree of $f$ is even and $\Char K \neq 2$. Assume that for any $n$ and any $\alpha \in f^{-n}(t)$ the splitting field of $f(x)-\alpha$ over $K(\alpha)$ has Galois group $S_d$. If the discriminants $\disc(f^{\circ n}(x) - t)$ are linearly independent in the $\mathbb{F}_2$ vector space $K^\times/K^{\times 2}$, then the arboreal representation attached to $f$ and $t$ is surjective.
\end{theorem}

This theorem allows us to prove surjectivity of some arboreal representations attached to even degree polynomials. The following conjecture was made by Odoni in \cite{Odoni1985a}*{Conjecture~7.5.}.

\begin{conjecture}
Let $F$ be a Hilbertian field. Then for any integer $d$ there exists a degree $d$ monic polynomial $f$ such that the associated arboreal representation (for $t=0$) is surjective.
\end{conjecture}

Until recently, the conjecture was open even for $F=\mathbb{Q}$ (see \cite{Jones2013}*{Conjecture~2.2}). Looper \cite{Looper2019} proved Odoni's conjecture for $F = \mathbb{Q}$ and $d$ prime, and showed that for many other values of $d$ Odoni's conjecture can be deduced from Vojta's conjecture. In Section \ref{surj} we prove the following theorem unconditionally.

\begin{theorem}
Let $d \geq 20$ be an even number. Then there exists infinitely many polynomials $f \in \mathbb{Q}[x]$ of degree $d$ such that the arboreal representation associated to $f$ and $t=0$ is surjective.
\end{theorem}

Recently, independently of our work, Joel Specter \cite{Specter2018-preprint} and Robert Benedetto and Jamie Juul \cite{Benedetto-Juul2019} proved related results on Odoni's conjecture for number fields. Specter shows  that Odoni's conjecture holds for any number field. Benedetto and Juul prove Odoni's conjecture for even degree polynomials over an arbitrary number field, and for odd degree polynomials over number fields $F$ that do not contain $\mathbb{Q}(\sqrt{d}, \sqrt{d+1})$.

Consider the case when $K=F(u)$ for some field $F$ and indeterminate $u$, $f$ is a polynomial defined over $F$, and $t=u$. Let $K_n$ be the splitting field of $f^{\circ n}(x)-t$ over $K$. Then the groups $\Gal(K_n/K)$ are the monodromy groups of the ramified coverings $\mathbb{P}^1_F \to \mathbb{P}^1_F$ induced by $f^{\circ n}$.   In this case the image of the resulting arboreal representation is known as the iterated monodromy group of $f$. These monodromy groups have been studied in the case when $f$ is post-critically finite  (see \cite{Nekrashevych 2011}), when $f$ is quadratic (see \cite{Pink2013}) and also in \cite{Juul2016-preprint}. Recall that a polynomial $f$ is called post-critically finite if for every root $\gamma$ of $f'$ the orbit of $\gamma$ under $f$ is finite. In Section \ref{iterated} we prove the following theorem.

\begin{theorem}\label{intro-monodromy}
Let $f \in F[x]$ be a degree $2m$ polynomial such that the Galois group of the splitting field of $f(x)-t$ over $K=\overline{F}(t)$ is the symmetric group $S_{2m}$. Assume that $f'(x)$ is irreducible over $F$ and $f$ is not post-critically finite. Then the arboreal representation attached to $f$ and $t$ over $K$ is surjective. 
\end{theorem}

In fact we prove a stronger statement that provides a criterion for the surjectivity of the iterated monodromy group attached to an even degree polynomial $f$ in terms of the critical orbit of $f$. Jamie Juul \cite{Juul2016-preprint}*{Proposition~3.2.} proves a version of theorem \ref{intro-monodromy} for finite level arboreal representations and polynomials of arbitrary degree. However we do not know if the class of polynomials for which Juul's result implies surjectivity of the infinite level arboreal representation is large. Our result shows that when $F$ is a number field most polynomials have surjective infinite level iterated monodromy groups, see Remark \ref{most-polys}.    


\section{Large arboreal representations over an arbitrary field}\label{large}

We begin by introducing some arboreal notation. Fix a field $K$ of characteristic not $2$, a polynomial $f \in K[x]$ of degree $d$, and an element $t \in K$. Assume $t$ is not periodic under $f$, and assume that the polynomials $f^{\circ n}(x)-t$ are separable for all $n$. Write $T_\infty$ for the preimage tree of $t$ under $f$ as in the introduction. Let $K_n/K$ denote the splitting field of $f^{\circ n}(x)-t$ over $K$. Denote by $T_n$ the tree formed by the first $n$ levels of $T_\infty$.
The Galois group $\Gal(K_n/K)$ injects into the automorphism group of the tree $\Aut(T_n)$. We call the extension $K_n/K$ \slantsf{maximal} if $\Gal(K_n/K)=\Aut(T_n)$.
 
The aim of this section is to prove a criterion for surjectivity of arboreal representations $\phi: G \to \Aut(T_\infty)$. We establish surjectivity in three steps. Step one is to show that for all $n$ and $\alpha \in f^{-n}(t)$ the Galois group of the splitting field of $f(x)-\alpha$ over $K(\alpha)$ is $S_n$. Step two is to show that the splitting field of $f(x)-\alpha$ is disjoint from $K_n/K(\alpha)$. Step three is to show that the $d^n$ extensions of $K_n$ given by the splitting fields of $f(x)-\alpha_i$ for different $\alpha_i \in f^{-n}(t)$ are linearly disjoint. We show that steps two and three can be reduced to the arithmetic of forward orbits of $f$. 

The main idea is simple: any proper normal subgroup of $S_d$ is contained in $A_d$; therefore to show that some $S_d$ extensions are linearly disjoint it is enough to show that their (unique) quadratic subfields are linearly disjoint, as we will now prove.

\begin{lemma}\label{disjoint}
Let $L/K$ be a Galois extension of fields with Galois group $S_d$. Let $F$ denote the unique quadratic extension of $K$ contained in $L$. Let $M/K$ be an arbitrary Galois extension. Then the fields $L$ and $M$ are linearly disjoint over $K$ if and only if the fields $F$ and $M$ are linearly disjoint over $K$.
\end{lemma}
\begin{proof}
If $L$ and $M$ are disjoint, then so are $F$ and $M$. Assume that $F$ and $M$ are disjoint. The field extension $K \subset \left(L \cap M\right) \subset L$ corresponds to a normal subgroup of $S_d$. Any nontrivial normal subgroup of $S_d$ is contained in $A_d$ and therefore if  $L \cap M$ were not $K$ it would contain $F$. 
\end{proof}

\begin{lemma}\label{all_disjoint}
Let $L_i/K$ for $i=1,...,n$ be $S_d$-Galois extensions of $K$. Let $F_i \subset L_i$ denote the unique quadratic extension of $K$ inside $L_i$. Then the fields $L_i$ are linearly disjoint if and only if the fields $F_i$ are linearly disjoint.
\end{lemma}
\begin{proof}
If $L_i$ are linearly disjoint, then $F_i$ are also linearly disjoint. Assume that $F_i$ are linearly disjoint. 
We use induction on $n$. The case $n=2$ follows from Lemma \ref{disjoint}. Suppose that the result holds for $n-1$. Then the fields $L_1, ..., L_{n-1}$ are linearly disjoint. Let $M$ denote the compositum $L_1...L_{n-1}$. The Galois group $\Gal(M/K)$ is isomorphic to $\left(S_d\right)^{n-1}$. We need to show that $M$ and $L_n$ are disjoint. By Lemma \ref{disjoint} it is enough to show that $F_n$ and $M$ are disjoint. 

Let $F/K$ be a quadratic extension of $K$ contained in $M$. Then $F$ defines a character $\chi: \left(S_d\right)^{n-1}=\Gal(M/K) \to \left\{ \pm 1\right\}$. The character $\chi$ can be written as
 \[\chi(a_1, ..., a_{n-1})=\chi_1(a_1) \cdot ... \cdot  \chi_{n-1}(a_{n-1}), \: (a_1, ..., a_{n-1})\in S_d \times ... \times S_d,\] 
 where $\chi_i : S_d \to \left\{ \pm 1\right\}$ are quadratic characters. The group $S_d$ has only two quadratic characters: the trivial character and the sign character. Let $b_i \in K$ be elements such that $F_i=K(\sqrt{b_i})$. Let $I \subset \{1, ..., n-1\}$ be the set of indices $i$ for which $\chi_i$ is nontrivial. Then the extension $F$ is obtained by adjoining to $K$ a square root of $\prod_{i \in I} b_i$.  In particular, since $F_n$ is disjoint from the compositum $F_1...F_{n-1}$, it is not a subfield of $M$ and therefore $M$ and $L_n$ are disjoint.
\end{proof}

Lemma \ref{all_disjoint} allows us to reduce proving surjectivity of an arboreal representation to Kummer theory.  We also need a standard fact about discriminants.

\begin{lemma}\label{products}
Let $S$ be a commutative ring, let $P,Q \in S[x]$ be arbitrary polynomials. Then $\disc (P Q)=(-1)^{\deg P \deg Q}\disc(P)\disc(Q)R(P,Q)^2$ where $R(P,Q)$ is the resultant of $P$ and $Q$.
\end{lemma}
\begin{proof}
We recall the normalization of discriminant and resultant for non-monic polynomials from \cite{LangAlgebra}*{Chapter~IV~\S8}.
Assume $P,Q$ have degrees $n,m$, leading coefficients $a,b$, and roots $\alpha_1, ..., \alpha_n$ and $\beta_1, ..., \beta_m$. Then the resultant is given by $R(P,Q)=a^mb^n \prod_{i,j}(\alpha_i-\beta_j)$, and the relevant discriminants are given by the formulas $\disc(P)=(-1)^{n(n-1)/2}a^{2n-2}\prod_{i \neq j}(\alpha_i-\alpha_j)$, $\disc(Q)=(-1)^{m(m-1)/2}b^{2n-2}\prod_{i \neq j}(\beta_i-\beta_j)$, and \[\disc(PQ)=(-1)^{(n+m)(n+m-1)/2}(ab)^{2(n+m)-2} \prod_{i\neq j}(\alpha_i-\alpha_j) \prod_{i \neq j}(\beta_i-\beta_j)\prod_{i,j}(\alpha_i-\beta_j)^2.\]
\end{proof}

Fix a field $K$ of characteristic not $2$ and fix  an element  $t \in K$. Let $f \in K[x]$ be a degree $d$ polynomial where $d$ is even. Consider the corresponding tree $T_\infty$, arboreal representation $\phi : G \to \Aut(T_\infty)$, and the extensions $K_n/K$.
\begin{theorem}\label{localtoglobal}
 Assume that for some $n \in \mathbb{Z}_{>0}$ the field extension $K_{n-1}/K$ is maximal.  Assume also that for every $\alpha \in f^{-(n-1)}(t)$ the polynomial $f(x)-\alpha$ is irreducible over $K(\alpha)$ and has Galois group $S_d$ or $A_d$. Then $K_n/K$ is maximal if and only if $\disc(f^{\circ n}(x) - t) \not\in K_{n-1}^{\times 2}$.
\end{theorem}
\begin{proof}
Assume that $K_n/K$ is maximal. Then the Galois group of $K_n/K_{n-1}$ acts on the roots of $f^{\circ n}(x)-t$ as a product of $d^{n-1}$ copies of $S_d$. In particular there is an element of $\Gal(K_n/K_{n-1})$ that induces an  odd permutation of the roots of $f^{\circ n}(x)-t$. Therefore $\disc(f^{\circ n}(x)-t)$ is not a square in $K_{n-1}$.

Now assume $\disc(f^{\circ n}(x) - t) \not\in K_{n-1}^{\times 2}$. For any choice of the root $\alpha$ of $f^{\circ n-1}(x)-t$  consider the discriminant $\disc(f(x)-\alpha) $. Assume that $\disc (f(x) - \alpha)$ is a square in $K_{n-1}$. Since the extension $K_{n-1}/K$ is maximal, the Galois action on the roots $\alpha_i \in f^{-(n-1)}(t)$ is transitive. Therefore for every $\alpha_i$ the discriminant $\disc(f(x)-\alpha_{i})$ is in $K_{n-1}^{\times 2}$ as well. By Lemma \ref{products} the following identity holds
 \begin{IEEEeqnarray*}{rCl}
 \prod_{\alpha \in f^{-(n-1)}(t)} \disc\left(f(x)-\alpha \right)\ &\equiv& \disc \left(\prod_{\alpha \in f^{-(n-1)}(t)} \left(f(x)-\alpha \right)\right) \\
&\equiv& \disc(f^{\circ n}(x)-t) \not\equiv 1 \pmod {K_{n-1}^{\times 2}}.
\end{IEEEeqnarray*}
 
 Therefore $\disc(f(x)-\alpha) \not\in K_{n-1}^{\times 2}$ and the Galois group of $f(x)-\alpha$ over $K(\alpha)$ is equal to $S_d$. By Lemma \ref{disjoint}, applied to the splitting field of $f(x)-\alpha$ over $K(\alpha)$ and $K_{n-1}/K(\alpha)$, the polynomial $f(x)-\alpha$ is irreducible over $K_{n-1}$ and has the Galois group equal to $S_d$. In order to apply Lemma \ref{all_disjoint} we need to show that $\disc(f(x)-\alpha_{i})$ are linearly independent in the $\F_2$ vector space $K_{n-1}^\times/K_{n-1}^{\times 2}$. 
 
 We claim that for $m \leq n-1$ and $\beta_i \in f^{-m}(t)$ the elements $\disc(f^{n-m}(x)-\beta_i)$ are linearly independent in the $\F_2$ vector space $K_{n-1}^{\times}/K_{n-1}^{\times 2}$. We prove the statement by induction. When $m=0$ the statement holds by the assumption of the theorem. Assume the statement is true for $m=l-1$. Assume that there is a nonempty subset $J \subset f^{-l}(t)$ such that $\prod_{\beta \in J} \disc(f^{\circ n-l}(x)-\beta)$ is a square. Say that $\beta_i$ and $\beta_j$ belong to the same cluster if $f(\beta_i)=f(\beta_j)$ (in other words $\beta_i$ and $\beta_j$ have the same parent in the tree). If every cluster of roots is either contained in $J$ or has an empty intersection with $J$, then $f^{-1}(f(J))=J$, and so by Lemma \ref{products}
 \begin{IEEEeqnarray*}{rCl}
 	\prod_{\beta \in J} \disc(f^{\circ n-l}(x)-\beta) &\equiv& \prod_{\gamma \in f(J)} \prod_{\beta \in f^{-1}(\gamma)} \disc(f^{\circ n-l}(x)-\beta) \\
 	&\equiv& \prod_{\gamma \in f(J)}\disc(f^{n-l+1}(x)-\gamma) \pmod {K_{n-1}^{\times 2}}.
\end{IEEEeqnarray*}
 The right hand side is not a square by the induction hypothesis. Therefore, we can assume that there is a cluster $I$ such that $I$ has a nontrivial intersection with $J$. Choose two elements $\beta', \beta'' \in I$ such that $\beta'$ is in $J$ and $\beta''$ is not in $J$. Since the extension $K_{n-1}/K$ is assumed to be maximal, there exists an element  $\sigma \in \Gal(K_{l}/K)$ that acts on the roots of $f^{\circ l}(x)-t$ as a transposition of $\beta'$ and $\beta''$. Then 
\begin{IEEEeqnarray*}{rCl}
1 &\equiv& \left(\prod_{\beta \in J} \disc(f^{\circ n-l}(x)-\beta) \right) \cdot \sigma \left( \prod_{\beta \in J} \disc(f^{\circ n-l}(x)-\beta)\right) \\
 &\equiv& \disc(f^{\circ n - l}(x)-\beta')\disc(f^{\circ n - l}(x)-\beta'') \pmod {K_{n-1}^{\times 2}}.
\end{IEEEeqnarray*}

Since  $\Gal (K_l/K)$ acts doubly-transitively on the cluster, for every pair of elements $\beta_i, \beta_j \in I$ the product $\disc(f^{\circ n - l}(x)-\beta_i)\disc(f^{\circ n - l}(x)-\beta_j) $ is a square in $K_{n-1}$. Since $d$ is even, the product $\prod_{\beta \in I} \disc(f^{\circ n-l}-\beta)$ is a square in $K_{n-1}$. By Lemma \ref{products} we have the identity 
\[\prod_{\beta \in I} \disc(f^{\circ n-l}(x)-\beta) \equiv \disc(f^{n-l+1}(x) - f(\beta) ) \pmod{ K_{n-1}^{\times 2}},\]
where the right hand side is not a square by the induction hypothesis. Therefore the elements $\disc(f^{\circ n-l}- \beta_i)$ are linearly independent $\bmod K_{n-1}^{\times 2}$.

Now we can apply Lemma \ref{all_disjoint} to the extensions given by the splitting fields of $f(x)-\alpha_i$ over $K_{n-1}$. We have proved that each extension is an $S_d$ extension. The unique quadratic subextension of this splitting field is $K_{n-1}\left(\sqrt{\disc\left(f(x)-\alpha_{i}\right)}\right)$. Since $\disc(f(x)-\alpha_{i})$ are linearly independent in $K_{n-1}/K_{n-1}^{\times 2}$, the corresponding quadratic extensions are linearly disjoint.

 \end{proof}

 \begin{proposition}\label{tree_quadratics}
 Assume that the extension $K_n/K$ is maximal. Then the quadratic subextensions $K \subset F \subset K_{n}$ are contained in the compositum of quadratic extensions $K(\sqrt{\disc(f^{\circ m}(x) - t)})$ for $m=1, ..., n$.
 \end{proposition}
 \begin{proof}
The group $\Aut(T_n)$ fits into an exact sequence
 \begin{equation}\label{exact}
 1 \to \left(S_d\right)^{d^{n-1}} \to \Aut(T_n) \to \Aut(T_{n-1}) \to 1.\tag{$\star$}
\end{equation}
Let $s_{m}: \Aut(T_n) \to S_{d^m}$ be the homomorphism given by the action of the automorphism group on $I_m$ (the $m$'th level of the tree). For $\sigma \in S_d$ let $\sigma_i$ be the element $(1,..., 1, \sigma, 1, ..., 1) \in S_d^{d^{n-1}}$ with $\sigma$ inserted in position $i$. Let $g \in \Aut(T_n)$ be an element such that $s_{n-1}(g)i=j$. Then $g \sigma_i g^{-1}$ is conjugate to $\sigma_j$ in $S_d^{d^{n-1}}$. Let $\chi: \Aut(T_n) \to \{\pm 1\}$ be an arbitrary homomorphism. Consider the restriction of $\chi$ to $S_d^{d^{n-1}}$. The map $\chi_i : S_d \to \{\pm 1\}$ given by $\chi_i(\sigma)=\chi(\sigma_i)$ is a quadratic character of $S_d$ which is either the sign character or the trivial character. Since $\sigma_i$ and $\sigma_j$ are conjugate in $\Aut(T_n)$ for all $i,j$, the characters $\chi_i$ and $\chi_j$ are equal for all $i,j$. Therefore the restriction of $\chi$ to $S_d^{d^{n-1}}$ is either the trivial character or the product of sign characters. Applying this inductively we can describe all quadratic characters of $\Aut(T_n)$.

 We claim that any quadratic character $\chi: \Aut(T_n) \to \{\pm 1\}$ is a product of characters $\hat{\chi}_m=\sign \circ s_m$, $m=1,...,n$. We prove the claim by induction. The case $n=0$ is trivial. Assume the statement is proved for $n-1$. Let $\chi: \Aut(T_n) \to \{\pm 1\}$ be a character. Consider the exact sequence \eqref{exact}. The restriction of $\chi$ to $S_d^{d^{n-1}}$ is either trivial or is equal to the restriction of $\hat{\chi}_{n}$. Therefore $\chi$ or $\chi \cdot \hat{\chi}_{n}$ descends to a character of $\Aut(T_{n-1})$ which is a product of $\hat{\chi}_m$'s by the induction hypothesis.
 
Quadratic subextensions of $K_n/K$ correspond to quadratic characters of $\Aut(T_n)$. The quadratic character corresponding to the extension $K(\sqrt{\disc(f^{\circ m}(x) - t)})$ is $\hat{\chi}_m$. Since any quadratic character is  a product of $\hat{\chi}_m$'s, any quadratic extension is contained in the compositum of $K(\sqrt{\disc(f^{\circ m}(x) - t)})$.
 \end{proof}
 
 The following corollary implies Theorem \ref{intro-surjective}.
 
\begin{corollary}\label{surjective}
	Let $f \in K[x]$ be a polynomial of even degree $d$. Assume that for every $k$ and every $\alpha \in f^{-k}(t)$ the Galois group of the splitting field of $f(x)-\alpha$ over $K(\alpha)$ is either $S_d$ or $A_d$. Then the arboreal representation $\phi_n: G \to \Aut \left(T_n\right)$ associated to $f$ is surjective if and only if the elements $\disc(f^{\circ i}(x)-t)$, $i=1,...,n$ are multiplicatively independent in $K^{\times}/K^{\times 2}$.
\end{corollary}
\begin{proof}
	We prove the statement using induction on $n$. The case $n=0$ is trivial. Assume $\Gal(K_{n-1}/K)=\Aut(T_{n-1})$. Then by Theorem \ref{localtoglobal} the equality $\Gal(K_n/K)=\Aut(T_n)$ holds if and only if $\disc(f^{\circ n}(x)-t) \not \in K_{n-1}^{\times 2}$. Since $\disc(f^{\circ n}(x)-t)$ is an element of $K$ the condition $\disc(f^{\circ n}(x)-t) \not \in K_{n-1}^{\times 2}$ holds if and only if the extension $K(\sqrt{\disc(f^{\circ n}(x)-t)})$ is not a subfield of $K_{n-1}$. By Proposition \ref{tree_quadratics} any quadratic subextension of $K_{n-1}/K$ is inside the field obtained by adjoining to $K$ square roots of $\disc(f^{i}(x)-t)$ for $i=1, ..., n-1$. Therefore the field $K(\sqrt{\disc(f^{\circ n}(x)-t)})$ is contained in $K_{n-1}$ if and only if $\disc(f^{n}(x)-t)$ is in the span of $\disc(f^{i}(x)-t), i=1,...,n-1$ in $K^\times/K^{\times 2}$. \qed
\end{proof}

\begin{remark}
Theorem \ref{localtoglobal}, Proposition \ref{tree_quadratics} and Corollary \ref{surjective} can be stated as results about maximal subgroups of the group $\Aut(T_n)$ without any reference to field theory.  
\end{remark}
 We record here the formula for the discriminant $\disc(f^{\circ n}(x)-t)$. The proof can be found in \cite{Odoni1985a}*{Lemma~3.1}.
 
 \begin{lemma}\label{disc}
 Let $f \in K[x]$ be a polynomial of even degree. Choose an algebraic closure $\Kbar$ of $K$. Let $\lambda_i$ for $i \in I$ be the roots of $f'(x)$ and $m_i$ be the corresponding multiplicities. Then for all $t \in K$, 
 \[\disc(f^{\circ n}(x)-t) \equiv \prod_{i \in I} (f^{\circ n}(\lambda_i)-t)^{m_i} \pmod {K^{\times 2}}\]
 \end{lemma}
 
 \section{Iterated monodromy groups}\label{iterated}
 
 Let $f \in K[x]$ be a polynomial and $t \in K(t)$ be an indeterminate. Then the arboreal representation associated to the preimages of $t$ under $f$ is known as the iterated monodromy group of $f$; see \cite{Nekrashevych2011}.  Using Corollary \ref{surjective} we can show that, under mild hypotheses, the iterated monodromy group is maximal for even degree polynomials. 
 
 \begin{theorem}\label{monodromy}
 	Let $f \in K[x]$ be an even degree polynomial and let $t \in K(t)$ be an indeterminate. Assume that the Galois group of  the splitting field of $f(x) - t$ over $K(t)$ is either $A_d$ or $S_d$. Choose a separable closure $\Ksep$. Let $\Gamma \subset \Ksep$ be the multiset of roots of $f'(x)$ and denote by  $\Gamma_n$ the multiset $f^{\circ n}(\Gamma)$. Let $\Gamma_n'$ be the sub(multi)set of $\Gamma_n$ which consists of the elements that have odd multiplicity in $\Gamma_n$ and assigns to them multiplicity one. Then the arboreal representation attached to $f$ and $t$ is surjective if and only if $\Gamma_n' \not\subset \bigcup_{i<n} \Gamma_i' $ for all $n$.
 \end{theorem}
 \begin{proof}
 	For every $n \in \mathbb{Z}_{\geq 0}$ any $\alpha \in f^{-n}(t)$ is transcendental over $K$. Therefore there is an isomorphism of fields $K(\alpha) \simeq K(u)$ for an indeterminate $u$. Since $f(x) - t$ is irreducible over $K(t)$ and has Galois group $S_d$ or $A_d$, the same is true for $f(x) - \alpha$ over $K(\alpha)$. By Lemma \ref{disc} applied to the field $K(t)$ the odd multiplicity roots of the polynomial $\disc(f^{\circ n}(x)-t) \in K[t]$ are precisely the elements of $\Gamma_n'$. 
 	
 	Assume that for every $n$ the set $\Gamma_n'$ contains an element that is not in the union of $\Gamma_i'$ for $i<n$. Then the polynomials $\disc(f^{\circ n}(x)-t)$ are linearly independent modulo $K(t)^{\times 2}$. Therefore, by Corollary \ref{surjective} the arboreal representation associated to $f$ and $t$ is surjective. 
 	
 	Assume that $\Gamma_n' \subset \bigcup_{i<n} \Gamma_i'$ for some $n$. Since $f(\Gamma_i') \subset \Gamma_{i+1}'$, for every $m>n$ the set $\Gamma_m'$ is contained in $\bigcup_{i<n} \Gamma_i'$. Hence for some $i \neq j$ the sets $\Gamma_i'$ and $\Gamma_j'$ are equal. In this case $\disc(f^{\circ i}(x)-t) \equiv \disc(f^{\circ j}(x)-t) \pmod {K(t)^{\times 2}}$ and the arboreal representation is not surjective by Corollary \ref{surjective}. 
 	\qed
 \end{proof}
 
 \begin{remark}
Assume that  $f(x)$ has coefficients in a subfield $k \subset K$ and $f'(x)$ is irreducible over $k$. Then $\Gal(\ksep/k)$ acts transitively on the multiset $\Gamma_n$. Since $d-1$ is odd, every element of $\Gamma_n$ has odd multiplicity, so the support of the multiset $\Gamma_n$ is $\Gamma_n'$. If for some $n \neq m$ the sets $\Gamma_n'$ and $\Gamma_m'$ intersect, then $\Gamma_n'=\Gamma_m'$. Therefore the conditions of Theorem \ref{monodromy} are satisfied if $f(x)-t$ has Galois group $A_d$ or $S_d$ and $f$ is not post-critically finite (the orbit of the set of critical points is infinite). This is the form in which Theorem \ref{intro-monodromy} was stated. 
 \end{remark}
 
 \begin{remark}\label{most-polys}
 If $f$ is an indecomposable polynomial of degree $>31$ over a field of characteristic $0$, then the Galois group of the splitting field of $f(x)-t$ is either $S_d, A_d,$ a cyclic group or a dihedral group \cite{Muller1995}; see \cite{Gurlanick-Saxl1995} for a positive characteristic version. Moreover, the Galois group can be cyclic or dihedral only when $f$ is conjugate to a Dickson polynomial; see \cite{Turnwald1995}*{Theorem~3.11}. Together with the previous remark this shows that when $K$ is a number field most polynomials of even degree satisfy the conditions of Theorem \ref{monodromy}.

 \end{remark}
 
 
 \section{Surjective arboreal representations over $\mathbb{Q}$}\label{surj}
In this section we assume $t=0$.  We can use the results from Section \ref{large} to construct surjective even degree arboreal representations over $\mathbb{Q}$. In order to do so we first need a way of showing that the ``tiny'' extensions given by the splitting fields of $f(x)-\alpha$ over $K(\alpha)$ for $\alpha \in f^{-n}(0)$ are maximal. In Proposition \ref{big_local} we will show that this can be achieved by combining local arguments at two primes.

We use the following convention: if $L$ is a field with a discrete valuation $v$, and $f \in L[x]$ is a polynomial, then the Newton polygon of $f$ is considered with respect to the equivalent valuation $v'$ that satisfies $v'(L)=\mathbb{Z}$.

\begin{lemma}\label{Eisenstein}
Let $K$ be a field with a discrete valuation $v$, and let $f \in K[x]$ be an Eisenstein polynomial of degree at least $2$. Then for any $n \in \mathbb{Z}_{\geq 0}$, any $\alpha_n \in f^{-n}(0)$, and any extension of $v$ to $K(\alpha_n)$ the polynomial $f(x)-\alpha_n$ is Eisenstein over $K(\alpha_n)$.
\end{lemma}
\begin{proof}
Using induction, if $g(x)\colonequals f(x)-\alpha_n$ is Eisenstein and $\alpha_{n+1}$ is a root of $g$,  then $\alpha_{n+1}$ has the minimal valuation in $K(\alpha_{n+1})$, and thus $f(x)-\alpha_{n+1}$ is Eisenstein over $K(\alpha_{n+1})$.
\end{proof}

\begin{lemma}\label{wild}
Let $p$ be a prime number and let $v$ denote the p-$adic$ valuation on the field $\mathbb{Q}$. Let $f \in \mathbb{Q}[x]$ be an irreducible polynomial of degree $d \geq 2$.  Assume that for some odd integer $k>0$ the Newton polygon of $f$ consists of two line segments: one from $(0,2)$ to $(k,0)$ and one from $(k,0)$ to $(d,0)$. Assume that for every $n \in \mathbb{Z}_{\geq 0}$ and every $\alpha \in f^{-n}(0)$ the polynomial $f(x)-\alpha$ is irreducible over $\Q(\alpha)$. Then for any $n \in \mathbb{Z}_{\geq 0}$ and any $\alpha_n \in f^{-n}(0)$ there is an extension of $v$ to $\Q(\alpha_{n})$ such that the polynomial $f(x)-\alpha_{n}$ has the same Newton polygon over $\mathbb{Q}(\alpha_n)_v$ as $f$ has over $\mathbb{Q}_p$.
\end{lemma}
\begin{proof}
We prove the statement by induction. Assume that $v_{n-1}$ is a discrete valuation on $\Q(\alpha_{n-1} )$ such that $g(x)=f(x)-\alpha_{n-1}$ has the Newton polygon of the desired shape. Let $\alpha_n$ be any root of $g$. Consider the extension $v_n$ of $v_{n-1}$ to $\Q(\alpha_{n})$ such that $v_n(\alpha_n)=2/k$. Since the local extension $\Q(\alpha_n)_{v_n}/\Q(\alpha_{n-1})_{v_{n-1}}$ has degree $k$ and $v_n(\alpha_n)=2/k$, the ramification index of $\Q(\alpha_n)/\Q(\alpha_{n-1})$ at $v_n$ is $k$. Therefore the valuation $v_n'=kv_n$ satisfies $v_n'(\Q(\alpha_n))=\mathbb{Z}$ and hence the Newton polygon of $f(x)-\alpha_n$ at $v_n'$ has the two-segment shape as claimed.
\end{proof}

\begin{proposition}\label{big_local}
Let $f \in \Q[x]$ be a polynomial of even degree $d$. Let $p$ be a prime satisfying $d/2 < p < d-2$, and let $q \neq p$ be an arbitrary prime. Assume that $f$ is Eisenstein at $q$ and that the Newton polygon of $f$ at $p$ is the union of two segments: one from $(0,2)$ to $(p,0)$ and one from $(p,0)$ to $(d,0)$. Then for every $n \in \mathbb{Z}_{\geq 0}$ and every $\alpha \in f^{-n}(0)$ the Galois group of the splitting field of $f(x)-\alpha$ over $\Q(\alpha)$ is either $A_d$ or $S_d$.
\end{proposition}
\begin{proof}
Fix $n$ and $\alpha \in f^{-n}(0)$. Since $f$ is Eisenstein at $q$, by Lemma \ref{Eisenstein} $f(x)-\alpha$ is irreducible (and, in particular, separable). Let $G$ be the Galois group of the splitting field of $f(x)-\alpha$ over $\Q(\alpha)$ and $r: G \to S_d$ be the transitive permutation representation associated to the roots of $f(x)-\alpha$.  By Lemma \ref{wild} there is an extension $v$ of the $p$-adic absolute value to $\Q(\alpha)$ such that the Newton polygon of $f(x)-\alpha$ contains a segment $(0,2)$ to  $(p,0)$. In particular, the splitting field of $f(x)-\alpha$ is wildly ramified at $v$ and therefore there exists $g \in G$ of order $p$. Since $p>d/2$ the permutation $r(g)$ is a cycle of length $p$. Since the length of the cycle $r(g)$ is greater than $d/2$, the permutation group $r:G \to S_n$ is primitive. A theorem of Jordan (see \cite{Dixon-Mortimer1996}*{Theorem~3.3E}) states that any primitive permutation group that contains a cycle fixing at least three points is either $A_d$ or $S_d$.

\end{proof}

We want to construct examples of polynomials over $\mathbb{Q}$ with surjective arboreal representation. Corollary \ref{surjective} breaks this problem into two: showing that the extensions given by the splitting fields of $f(x)-\alpha_i$ over $\Q(\alpha_{i})$ are $S_d$ or $A_d$ and showing that the discriminants $\disc(f^{\circ n}(x))$ are independent modulo squares. We use Proposition \ref{big_local} for the former problem. To deal with the latter we consider the case when $f'(x)=(x-C)g(x)^2$. In this case Lemma \ref{disc} shows that $\disc(f^{\circ n}(x))$ equals $f^{\circ n}(C)$ modulo squares. To show that $f^{\circ n}(C)$ are distinct modulo squares we will need to know that they have enough distinct prime divisors. For this the following lemma is useful.

\begin{lemma}\label{rigid}
Assume $f \in \mathbb{Q}[x]$ is integral at a prime $p$, and $f(0)=f(f(0))$. Assume that for some positive integers $n>k$ and for some $p$-integral $c \in \mathbb{Z}_{(p)}$ both $f^{\circ k}(c)$ and $f^{\circ n}(c)$ are divisible by $p$. Then $p$ divides $f(0)$.
\end{lemma}

\begin{proof}
Reducing modulo $p$ gives $0\equiv f^{\circ n}(c) \equiv f^{\circ n-k}(f^{\circ k}(c)) \equiv f^{\circ n-k}(0) \equiv f(0) \pmod p$.
\end{proof}

Now we give explicit examples of polynomials with surjective arboreal representation. Note that when $d\geq 20$ is an even number there exists a prime $p$ that satisfies $d - 3 \geq p \geq d/2 + 5$ (see for example \cite{Nagura1952}). This is why we have the condition $d \geq 20$ in the following theorem. 

\begin{theorem}
Fix an even integer $d\geq 20$ and a prime number $p$ satisfying $d-3 \geq p \geq d/2 + 5$. Let $k:=d/2-1$ and $u:=p-k-2$.  Given $A,B,C,D \in \mathbb{Q}$ consider the polynomial $f(x)$ such that $f'(x)=(2k+2)(x-C)(x^k+Ax^u+B)^2$ and $f(0)=D$.
Then for infinitely many choices of $(A,B,C,D)$ the arboreal representation associated to $f$ and $0$ is maximal.
\end {theorem}

\begin{proof}

The polynomial $f$ is given by the formula
\begin{multline}\label{poly}
f(x)=x^{2k+2}-\frac{2k+2}{2k+1}Cx^{2k+1}+\frac{2k+2}{k+u+2}2Ax^{k+u+2}-\frac{2k+2}{k+u+1}2ACx^{k+u+1} \\
+ \frac{2k+2}{2u+2}A^2x^{2u+2}- \frac{2k+2}{2u+1}CA^2x^{2u+1} + \frac{2k+2}{k+2}2Bx^{k+2} - \frac{2k+2}{k+1}2BC x^{k+1} + \frac{2k+2}{u+2}2AB x^{u+2} \tag{$\star$} \\
-\frac{2k+2}{u+1}2ABC x^{u+1} + (k+1)B^2 x^{2}-(2k+2)B^2Cx + D.
\end{multline}
The monomials are not necessarily ordered by decreasing degree, but the four largest degree and four smallest degree monomials are the first four and the last four respectively (this is one place where we use the condition $d-3 \geq p \geq d/2+5$).

We will use $s$-adic properties of the coefficients of $f$ for various primes $s$ to prove surjectivity of the associated arboreal representation. For this we will first choose some appropriate primes, and then use weak approximation to choose $A,B,C,D$, forcing $f$ to have a specific local behavior.

Choose a prime number $\ell > d^5, \ell \equiv 3 \pmod 4$. Choose another prime $q$ such that $q \equiv 1 \pmod \ell$ (in particular $q>\ell>d^5$). 
In order to apply Lemma \ref{rigid} we want $D$ to be a fixed point of $f$. The equation $f(D)=D$ is equivalent to the equation $U(A,B,D)- V(A,B,D)C =0$ where 
\begin{align}\label{U}
U(A,B, D)&=D^{2k+2} + \frac{2k+2}{k+u+2}2AD^{k+u+2} + \frac{2k+2}{2u+2}A^2D^{2u+2} + \frac{2k+2}{k+2}2BD^{k+2} \\
\notag &+\frac{2k+2}{u+2}2ABD^{u+2} + (k+1)B^2D^{2}
\end{align}
\begin{align}\label{V}
V(A,B, D)&=\frac{2k+2}{2k+1}D^{2k+1} + \frac{2k+2}{k+u+1}2AD^{k+u+1} + \frac{2k+2}{2u+1}A^2D^{2u+1} + \frac{2k+2}{k+1}2BD^{k+1} \\
\notag &+\frac{2k+2}{u+1}2ABD^{u+1} + (2k+2)B^2 D \
\end{align}
We choose $A,B,N_1 \in \mathbb{Z}$ satisfying the following local conditions.

\begin{description}
\item[prime $\ell$] Consider the quadratic equation $U(A,B, -p^2)+p^2 V(A,B, -p^2)=0$. The discriminant of this quadric $Q$ in $A,B$ is 

 \[-(2k+2)^2p^{4u+4t+8}\left( \frac{u^2(u+1)(12u^2+37u+29)}{2(2u+1)(2u+2)(u+1)^2(u+2)^2} \right)\]    
 which is a nonzero element of $\Q$. Since $\ell$ is large ($\ell>d^5$),  the discriminant of $Q$ is nonzero modulo $\ell$. Also (since $\ell$ is large) the quadric $V(A, B, -p^2)$ is not proportional to $Q$. Let $(A_\ell,B_\ell)$ be an $\F_\ell$ point on the (affine) conic $Q=0$ that does not lie on $V(A,B, -p^2)=0$.  Choose $N_1 \equiv p \pmod \ell$ and $(A,B) \equiv (A_\ell, B_\ell) \pmod \ell$.
 
 \item[prime $p$] Let $v_p(N_1)=v_p(A)=v_p(B)=1$. 
 
 \item[prime $q$] Let $v_q(A)=v_q(B)=1$ and $v_q(N_1)=0$. 
 
 \item[primes $<d$] For every prime $s$ less than $d$ and not equal to $2$  let $v_s(A)=v_s(B)=v_s(N_1)=v_s(d!)$. For $s=2$ assume that $v_2(A)=v_2(B)=0$. There exists a constant $M$ such that when $v_2(z)>M$ and $v_2(x)=v_2(y)=0$ the valuation of $U(x,y,x)/V(x,y,z)$ is \[v_2\left(\frac{(k+1)y^2z^2}{(2k+2)y^2z}\right)=v_2(z)-1.\] Let $v_2(N_1)=M$.

\end{description}
Let $K$ be the product of primes that are greater than $d$, not equal to $p,q,\ell$, and divide $N_1$. After raising $K$ to a sufficiently large power we can assume that $K$ is $1$ modulo every prime less than $d$ and also modulo $pq\ell$. The triple $(A,B,N_1  K^m)$ also satisfies the local conditions above for every sufficiently large $m$. Let $C(m):=U(A,B, -qN_1^2 K^{2m})/V(A,B, -qN_1^2K^{2m})$. Consider any prime $s|K$.  From the triangle inequality and formulas \eqref{U}, \eqref{V} there exists $M_0$ such that for $m>M_0$ we have $v_s(C(m))=2v_s(N_1K^m)$. Let $D(m)=-qN_1^2K^{2m}$. Consider the polynomial $f_m$ associated to $A,B, C(m), D(m)$ (given by the formula \eqref{poly}). When $m$ tends to infinity we have  $C(m) \sim \frac{2k+1}{2k+2}D,$ $f_m(C(m)) \sim -\frac{1}{2k+1}C(m)^{2k+2},$ $f_m^{\circ 2}(C(m)) \sim \frac{1}{(2k+1)^{2k+2}} C(m)^{(2k+2)^2}.$ Since $f'_m(x)=(x-C(m))(X^k+AX^u+B)^2$ the function $f_m$ is decreasing when $x \leq C(m)$ and increasing when $x \geq C(m).$ There exists $M_1$ such that for $m>M_1$ we have $C(m), f_m(C(m))<0$ and $f_m^{\circ 2}(C(m))>0>C(m)$ which implies that $f_m^{\circ n}(C(m))>0$ for all $n>1$. Finally fix an arbitrary $m>\max(M_0, M_1)$ and set $C=C(m)$, $D=D(m)$, $N=N_1K^m$.

Local conditions on $A,B,N$ imply local properties of $D=-qN^2$. We can also find valuations of $C=U(A,B,D)/V(A,B,D)$ using the non-archimidean triangle inequality. 

To summarize, we have chosen $A,B,N, D \in \mathbb{Z}$ and $C \in \mathbb{Q}$ such that the following conditions hold.

\begin{enumerate}
\item \label{fixed} $f(D)=D$, since $C=U(A,B,D)/V(A,B,D)$;

\item \label{D=-qN^2} $D=-qN^2$;

\item \label{C=D} $C$ is $\ell$-integral and $C \equiv D \pmod \ell$. Since $(A,B,N,q)\equiv (A_l, B_l, p, 1) \pmod \ell$ both $C$ and $D$ are $p^2$ modulo $\ell$;

\item \label{p} Valuations at the prime $p$ satisfy $v_p(D)=2$, $v_p(A)=v_p(B)=1$, $v_p(C)=2$;

\item \label{q} Valuations at the prime $q$ satisfy $v_q(A), v_q(B), v_q(C) \geq 1, v_q(D)=1$;

\item \label{S} There is a finite set of primes $S \subset \mathbb{N} \setminus \{1, ...,d \}$ such that $f$ has coefficients in $\Z[S^{-1}]$, $C$ is in $\Z[S^{-1}]$, and every prime in $S$ divides the denominator of $C$. Indeed our conditions for primes less than $d$ imply that the valuation of $C$ at these primes is nonnegative; 

\item \label{archim} $f(C)$ is negative, and $f^{\circ n}(C)$ is positive for $n>1$;

\item \label{C>D} For every prime $s \neq q$  dividing $D$ we have $v_s(C)\geq v_s(D)/2$. Indeed, primes dividing $D$ are primes dividing $pqK$ and primes less than $d$. For each of them the property holds. 
\end{enumerate}

We now show that conditions \ref{fixed} -- \ref{C>D} imply surjectivity of the arboreal representation.  Condition \ref{q} implies that $f$ is Eisenstein at $q$, therefore by Lemma \ref{Eisenstein} all iterates of $f$ are irreducible. Note that since $D \neq 0$, by condition \ref{fixed} the point $0$ is not periodic under $f$ and the arboreal representation is well-defined. Condition \ref{p} implies that the Newton polygon of $f$ at $p$ is a union two segments: one from $(0,2)$ to $(p,0)$ and one from $(p,0)$ to $(d,0)$. Hence by proposition \ref{big_local} for any $n$ and $\alpha \in f^{-n}(0)$ the Galois group of the splitting field of $f(x) - \alpha$ over $K(\alpha)$ is $A_d$ or $S_d$. Now consider the discriminants $\disc(f^{\circ n}(x))$ as elements of $\Q/\Q^{\times 2}$. If they are linearly independent, Corollary \ref{surjective} would imply that the arboreal representation is surjective. By Lemma \ref{disc} we have the equality $\disc(f^{\circ n}(x)) \equiv  f^{\circ n}(C) \pmod {\Q^{\times 2}}$. We claim that the numbers $f^{\circ n}(C)$ are independent modulo squares. Indeed $f(C)$ is negative by condition \ref{archim}, and therefore is not a square. Consider the number $f^{\circ n}(C)$ for some $n>1$. By condition \ref{C=D} we have the formula $f^{\circ n}(C) \equiv f^{\circ n}(D) \equiv D \equiv -N^2 \pmod \ell$. Therefore $f^{\circ n}(C)$ is not a square modulo $\ell$ since $\ell \equiv 3 \pmod 4$. By condition \ref{S} the number $f^{\circ n}(C)$ is $S$-integral. For any prime $s \in S$ when \eqref{poly} is evaluated at $x=C$ the valuation of the sum of the first two terms is smaller than the valuation of the rest of the terms. The non-Archimedean triangle inequality implies that $v_s(f(C))$ is equal to the $s$-adic valuation of $(x^{2k+2}-\frac{2k+2}{2k+1}Cx^{2k+1})|_{x=C}=-\frac{1}{2k+1}C^{2k+2}$ which is $(2k+2)v_s(C) < v_s(C)$. If $v_s(x)<v_s(C)$, then $v_s(f(x))=(2k+2)v_s(C)$. Therefore the $s$-adic valuation of $f^{\circ n}(C)$ is even and negative. Since $f^{\circ n}(C)$ is positive, has square denominator, and is not a square modulo  $\ell$, the numerator of $f^{\circ n}(C)$ has an odd multiplicity prime divisor $\gamma \not\in S$ that is not a square modulo $\ell$. If $v_\gamma(f^{\circ m}(C))$ is even for all $m<n$, then $f^{\circ n}(C)$ is linearly independent from $f^{\circ m}(C)$ when $ m<n$. Assume that $v_\gamma (f^{\circ n}(C))$ is odd. Then by Lemma \ref{rigid}, $\gamma$ divides $D$.  The prime $\gamma$ does not equal $q$, since $q \equiv 1 \pmod \ell$. Hence condition \ref{D=-qN^2} implies $\gamma$ divides $N$ and condition \ref{C>D} implies $v_\gamma(C) >0$. By the formula \eqref{poly} and the triangle inequality if $v_\gamma(x) \geq v_\gamma(D)/2$, then $v_\gamma(f(x))  \geq v_\gamma(D)$, and if $v_\gamma(x) \geq v_\gamma(D)$, then $v_\gamma(f(x))=v_\gamma(D)$. Therefore by condition \ref{C>D} the valuation $v_{\gamma}(f^{\circ n}(C))$ is equal to $v_{\gamma}(D)$ which is even, contradiction. Thus the numbers $f^{\circ k}(C)$ are independent modulo squares.

\end{proof}
\section*{Acknowledgements} 

I would like to thank my advisor Bjorn Poonen for careful reading of the paper and many helpful discussions. I also thank Dmitri Kubrak and Svetlana Makarova for comments on an earlier draft of the paper. I am grateful to an anonymous referee for suggesting to add an ``only if'' statement to Theorem \ref{monodromy}.

\end{document}